\documentclass[12pt,a4paper]{article}

\usepackage[english]{babel}
\usepackage{amsmath}
\usepackage{amsfonts}
\usepackage{amsthm}
\usepackage{amscd}
\usepackage{amssymb}
\usepackage{bm}
\usepackage[utf8]{inputenc}
\usepackage{textcomp}
\usepackage{enumerate}
\usepackage{graphicx}
\usepackage[margin=3.1cm]{geometry}

\usepackage{tikz}

\newtheoremstyle{mystyle}{}{}{\slshape}{2pt}{\scshape}{.}{ }{}
\newtheoremstyle{etapestyle}{}{}{\itshape}{2em}{\sffamily}{:}{ }{\thmname{#1}}
\newtheoremstyle{definitionstyle}{}{}{}{2pt}{\bfseries}{.}{ }{}

\newtheorem{thm}{Theorem}[section]

\newtheorem*{thm3}{Theorem 3}

\newtheorem*{lem2}{Lemma 2}
\newtheorem*{prob}{Problem}

\newtheorem{cor}[thm]{Corollary}

\newtheorem{prop}[thm]{Proposition}
\newtheorem{defi}[thm]{Definition}

\newtheorem{lemme}[thm]{Lemma}

\theoremstyle{mystyle}

 \theoremstyle{remark}

\theoremstyle{etapestyle}

\theoremstyle{definitionstyle}

\newcommand{\op}{\operatorname}
\newcommand{\rw}{\rightarrow}

\newcommand{\si}{\sigma}
\begin{document}

\title{Special reductive groups over an arbitrary field}

\bigskip

\author{Mathieu Huruguen}
\date{}

\maketitle

\bigskip

\begin{abstract}
A linear algebraic group $G$ defined over a field $k$
is called special
if every $G$-torsor over every field extension of $k$ is trivial. In 1958
Grothendieck classified
special groups in the case where the base field is
algebraically closed. In this paper
we describe the derived subgroup and the coradical of a special reductive group over an arbitrary field $k$. We also classify special semisimple 
groups, special reductive groups of inner type and special quasisplit reductive groups over an arbitrary field $k$.
\end{abstract}

\bigskip
\bigskip
\bigskip

\section{Introduction}
Let $k$ be a base field and $G$ an algebraic group defined over $k$.
The group $G$ is called \textbf{special} if every $G$-torsor defined over a field extension of $k$ is trivial. 
In other words, if for every field extension $K$ of $k$ the first fppf-cohomology set $H^1(K,G)$ contains only one element. 
Examples of special linear groups include 
the additive group $\mathbb{G}_a$, the multiplicative group 
$\mathbb{G}_m$, the general linear group $\op{GL}_n$, and more generally the group $\op{GL}_1(A)$, where $A$ is a central simple algebra over $k$, and
the classical groups $\op{SL}_n$ and $\op{Sp}_{2n}$. In contrast, the group $\op{SO}_{n}$ is not special for $n\geqslant3$. The 
special groups over an algebraically closed field were introduced by Serre in \cite{Esp} - recently reprinted in \cite{Esp2}. 
In this paper, Serre gave the basic properties of special groups, for example, he showed that they are linear and connected.
The study of special groups over an algebraically closed field was then completed by Grothendieck in \cite{Groth}. In the reductive case, his result can be stated as 
follows : 
\begin{thm}[Grothendieck, $1958$]\label{algclos}
Suppose that $G$ is reductive and $k$ is algebraically closed. Then $G$ is special if and only if its derived subgroup is isomorphic to a direct product 
$$G_1\times G_2\times\cdots\times G_r$$
where, for each $i$, the group $G_i$ is isomorphic to $\op{SL}_{n_i}$ or $\op{Sp}_{2n_i}$ for some integer $n_i$.
\end{thm}
The result of Grothendieck naturally raises the problem of classifying special reductive groups over an arbitrary field $k$.
The present paper is an attempt to solve this problem. Our most general classification result is the following :
\begin{thm}\label{classif}
Let $G$ be a reductive algebraic group over $k$. Then $G$ is special if and only if the three following conditions hold :
\begin{enumerate}[(1)]
 \item The derived subgroup of $G$ is isomorphic to  
$$R_{K_1|k}(G_1)\times R_{K_2|k}(G_2)\times\cdots\times R_{K_r|k}(G_r)$$
where, for each index $i$, the extension $K_i$ of $k$ is finite and separable, $R_{K_i|k}$ denotes the Weil scalar restriction 
functor - see for example \cite[Lemma $20.6$]{Inv} -  and the group $G_i$ is isomorphic over $K_i$ to either $\op{SL}_1(A_i)$, where 
$A_i$ is a central simple algebra over $K_i$, or $\op{Sp}_{2n_i}$ for some integer $n_i$.
 \item The coradical of $G$ is a special torus.
 \item For every field extension $K$ of $k$, we have 
 $$\op{Im}(\alpha_{G',K})+\op{Ker}(H^1(K,Z_{G'})\rw H^1(K,Z_{G}))=H^1(K,Z_{G'}).$$
where $Z_G$ is the center of $G$, $Z_{G'}$ is the center of its derived subgroup $G'$, and the map $\alpha_{G',K}$ is defined in \ref{alphabeta}. 
 \end{enumerate}
\end{thm}
Condition {\it (1)} above is explicit, as well as condition {\it (2)}, by the classification of special tori due to Colliot-Thélène and recalled at the end of the present paper. 
In contrast, condition {\it (3)} is not very explicit in general. 
However, under some additional assumptions on the group $G$, namely that $G$ is semisimple, an inner form of a Chevalley group, or quasisplit, 
we are able to make condition {\it (3)} completely explicit,
providing the classification in these cases. We hope that an explicit version of condition {\it (3)} will emerge in the future, unifying these cases and providing the classification 
of special reductive groups.

\medskip

The paper is organized as follows. In Section \ref{preliminary} we gather some facts to be used in the following sections. In Section \ref{derivedcoradical}
we determine which algebraic groups can arise as derived subgroups of a special group and which can arise as coradicals,
respectively in Proposition \ref{derived} and \ref{coradical}. In Section \ref{classification} we prove our general classification result stated above and then derive 
from it the classification of special semisimple groups, special reductive groups of inner type and special quasisplit groups in Proposition \ref{semisimple}, 
\ref{inner} and \ref{qsplit} respectively. Finally, we recall in Section \ref{tori} the classification of special tori due to Colliot-Thélène. 

\medskip

To finish this introduction, we say a word about special non-reductive groups. First, by \cite[Lemma $1.13$]{Sansuc}, if an algebraic group $G$ over a field $k$
possesses a $k$-split unipotent normal subgroup $U$, then $G$ is special if and only if $G/U$ is special. For example, if the field $k$ is perfect, then 
$G$ is special if and only if its quotient by the unipotent radical - which is a reductive group - is special, as every unipotent group over $k$ is $k$-split. 
On a different note, Nguyen classifies 
special 
unipotent groups over ``reasonable fields'' in \cite{Tan}. It is a direct consequence of the fact that the additive group $\mathbb{G}_a$ is special that every 
$k$-split unipotent group is special. In \cite{Tan}, Nguyen proves conversely that a special unipotent group is $k$-split for certain fields $k$, for example 
when $k$ is finitely generated over a perfect field.

\section*{Acknowledgment}
I would like to warmly thank Zinovy Reichstein for pointing out this problem to me, for very interesting discussions on this topic and also for remarks which helped 
improving 
the exposition of this paper. I would also like to thank Roland Lötscher for bringing reference \cite{Tan} to my attention.

%\textbf{Acknowledgement. }I would like to warmly thank Zinovy Reichstein for pointing this problem to me and also for very interesting discussions on this topic.

\section{Preliminary results}\label{preliminary}
Let $k$ be a base field and $G$ a reductive algebraic group defined over $k$. Throughout the paper we denote by 
 $Z_G$ the {\bf center} of $G$, by $G_{\op{ad}}$ the {\bf adjoint quotient} 
$G/Z_G$ of $G$, by $G'$ the {\bf derived subgroup} of $G$, by $R_G$ the {\bf radical} of $G$ and by $C_G$ its {\bf coradical}.
We have an exact sequence of algebraic groups : 

$$
\begin{tikzpicture}
\node (A) at (0,0) {$1$};
\node (B) at (1.5,0) {$Z_G$};
\node (C) at (3.4,0) {$G$};
\node (D) at (5.5,0) {$G_{\op{ad}}$};
\node (E) at (8,0) {$1 \quad \quad \quad  (*)$};
\draw[->,>=latex] (A) to (B);
\draw[->,>=latex] (B) to (C);
\draw[->,>=latex] (C) to (D);
\draw[->,>=latex] (D) to (E);
\end{tikzpicture}
$$

\begin{defi}\label{alphabeta}
Let $K$ be a field extension of $k$. We denote by 
$$\begin{tikzpicture}
\node (A) at (0,0) {$\alpha_{G,K} : G_{\op{ad}}(K)$};
\node (B) at (5.5,0) {$H^1(K,Z_G)\text{ and } \beta_{G,K} : H^1(K,G_{\op{ad}})$};
\node (C) at (10.7,0) {$H^2(K,Z_G)$};
\draw[->,>=latex] (A) to (B);
\draw[->,>=latex] (B) to (C);
\end{tikzpicture}$$
the connecting maps in fppf-cohomology obtained from the exact sequence $(*)$ above. 
\end{defi}

\begin{prop}\label{immediate}
The group $G$ is special if and only if for every field extension $K$ of $k$, the map $\alpha_{G,K}$ is surjective and the map $\beta_{G,K}$
has trivial kernel. 
\end{prop}
\begin{proof}
 Part of the exact sequence of pointed sets obtained from the exact sequence $(*)$ reads :
 $$ \begin{tikzpicture}

\node (A) at (-0.1,0) {$G_{\op{ad}}(K)$};
\node (B) at (3,0) {$H^1(K,Z_G)$};
\node (C) at (6,0) {$H^1(K,G)$};
\node (D) at (9,0) {$H^1(K,G_{\op{ad}})$};
\node (E) at (12.6,0) {$H^2(K,Z_G)$};

\node (F) at (1.3,0.2) {$\alpha_{G,K}$};
\node (G) at (10.8,0.25) {$\beta_{G,K}$};

\draw[->,>=latex] (A) to (B);
\draw[->,>=latex] (B) to (C);
\draw[->,>=latex] (C) to (D);
\draw[->,>=latex] (D) to (E);

\end{tikzpicture}$$
It is a straightforward consequence of the exactness of this sequence of pointed sets that $H^1(K,G)$ is trivial if and only if 
$\alpha_{G,K}$ is surjective and $\beta_{G,K}$ has trivial kernel.  
\end{proof}
  
\begin{prop}\label{plusdinfos}
Let $K$ be a field extension of $k$. If $G$ is special then the following properties hold :
\begin{enumerate}[(1)]
 \item the map 
$$ \begin{tikzpicture}
\node (A) at (0,0) {$\beta_{G',K} : H^1(K,(G')_{\op{ad}})$};
\node (B) at (4,0) {$H^2(K,Z_{G'})$};

\draw[->,>=latex] (A) to (B);
\end{tikzpicture}$$ has trivial kernel.
 \item the image of the map $\beta_{G',K}$ intersects the kernel of the morphism 
 $$ \begin{tikzpicture}
\node (A) at (0,0) {$H^2(K,Z_{G'})$};
\node (B) at (3,0) {$H^2(K,Z_G)$};

\draw[->,>=latex] (A) to (B);
\end{tikzpicture}$$
trivially.
 \item the following exact sequence of diagonalizable groups 
 $$ \begin{tikzpicture}

\node (A) at (0,0) {$1$};
\node (B) at (1.5,0) {$Z_{G'}$};
\node (C) at (3.4,0) {$Z_G$};
\node (D) at (5.5,0) {$Z_G/Z_{G'}$};
\node (E) at (8.3,0) {$1 \quad \quad \quad (**)$};

\draw[->,>=latex] (A) to (B);
\draw[->,>=latex] (B) to (C);
\draw[->,>=latex] (C) to (D);
\draw[->,>=latex] (D) to (E);

\end{tikzpicture}$$
induces the following exact sequence in fppf-cohomology : 
$$ \begin{tikzpicture}

\node (B) at (3,0) {$0$};
\node (C) at (5.5,0) {$H^1(K,Z_G/Z_{G'})$};
\node (D) at (9,0) {$H^2(K,Z_{G'})$};
\node (E) at (12.2,0) {$H^2(K,Z_G)$};

%\node (F) at (10.5,0.2) {$\gamma_K$};
%\node (G) at (1.5,0.2) {$\alpha_K$};

\draw[->,>=latex] (B) to (C);
\draw[->,>=latex] (C) to (D);
\draw[->,>=latex] (D) to (E);

\end{tikzpicture}$$
\end{enumerate}
\end{prop}

\begin{proof}
We will obtain these properties from Proposition \ref{immediate} and the fact that the adjoint groups $(G')_{\op{ad}}$ and $G_{\op{ad}}$ are equal.  
The inclusion of $G'$ in $G$ gives rise to a natural 
commutative diagram :
$$ \begin{tikzpicture}
\node (F) at (0,1.5) {$1$};
\node (G) at (1.5,1.5) {$Z_{G'}$};
\node (H) at (3.4,1.5) {$G'$};
\node (I) at (5.5,1.5) {$(G')_{\op{ad}}$};
\node (J) at (7.3,1.5) {$1$};

\draw[->,>=latex] (F) to (G);
\draw[->,>=latex] (G) to (H);
\draw[->,>=latex] (H) to (I);
\draw[->,>=latex] (I) to (J);

\node (A) at (0,0) {$1$};
\node (B) at (1.5,0) {$Z_G$};
\node (C) at (3.4,0) {$G$};
\node (D) at (5.5,0) {$G_{\op{ad}}$};
\node (E) at (7.3,0) {$1$};

\draw[->,>=latex] (A) to (B);
\draw[->,>=latex] (B) to (C);
\draw[->,>=latex] (C) to (D);
\draw[->,>=latex] (D) to (E);

\draw[->,>=latex] (G) to (B);
\draw[->,>=latex] (H) to (C);

\node (F) at (5.5,0.75) {$||$};

\end{tikzpicture}$$
where both rows are exact. 
The diagram above leads to a commutative diagram of connecting maps in fppf-cohomology :
$$ \begin{tikzpicture}
\node (A) at (0,1) {$H^1(K,G_{\op{ad}})$};
\node (B) at (4,2) {$H^2(K,Z_{G'})$};
\node (C) at (4,0) {$H^2(K,Z_G)$};

\node (E) at (2,1.85) {$\beta_{G',K}$};
\node (F) at (2,0.25) {$\beta_{G,K}$};
%\node (G) at (4.3,1) {$\delta_K$};
\draw[->,>=latex] (A) to (B);
\draw[->,>=latex] (A) to (C);
\draw[->,>=latex] (B) to (C);

\end{tikzpicture}$$
 where the vertical map is induced by the inclusion of $Z_{G'}$ in $Z_G$. As $G$ is special, by Proposition \ref{immediate}, the map $\beta_{G,K}$ has trivial kernel, which readily 
 implies {\it(1)} and {\it(2)}.
 
 \medskip
 
 To prove {\it (3)} we look at the following diagram in fppf-cohomology:
  $$ \begin{tikzpicture}
\node (A) at (0,1) {$G_{\op{ad}}(K)$};
\node (B) at (4,2) {$H^1(K,Z_{G'})$};
\node (C) at (4,0) {$H^1(K,Z_G)$};

\node (E) at (1.8,0.3) {$\alpha_{G,K}$};
\node (F) at (1.8,1.7) {$\alpha_{G',K}$};
\draw[->,>=latex] (A) to (B);
\draw[->,>=latex] (A) to (C);
\draw[->,>=latex] (B) to (C);

\end{tikzpicture}$$
where the vertical map is induced by the inclusion of $Z_{G'}$ in $Z_G$. 
As $G$ is special, by Proposition \ref{immediate} we see that $\alpha_{G,K}$ is surjective, forcing the vertical map to be surjective as well.
Now, part of the long exact sequence in fppf-cohomology obtained from the short exact sequence $(**)$ reads : 
$$ \begin{tikzpicture}

\node (A) at (0,0) {$H^1(K,Z_{G'})$};
\node (B) at (3,0) {$H^1(K,Z_G)$};
\node (C) at (6.3,0) {$H^1(K,Z_G/Z_{G'})$};
\node (D) at (9.6,0) {$H^2(K,Z_{G'})$};
\node (E) at (12.5,0) {$H^2(K,Z_G)$};

%\node (F) at (10.5,0.2) {$\gamma_K$};
%\node (G) at (1.5,0.2) {$\alpha_K$};

\draw[->,>=latex] (A) to (B);
\draw[->,>=latex] (B) to (C);
\draw[->,>=latex] (C) to (D);
\draw[->,>=latex] (D) to (E);

\end{tikzpicture}$$
and the result readily follows.  
\end{proof}

\begin{prop}\label{classif1}
Suppose that the coradical $C_G$ of $G$ is special. Then $G$ is special if and only if, for every field extension $K$ of $k$, the following two conditions hold :
\begin{enumerate}[(1)]
  \item $\op{Im}(\alpha_{G',K})+\op{Ker}(H^1(K,Z_{G'})\rw H^1(K,Z_{G}))=H^1(K,Z_{G'})$
   
 \item the map 
 $$ \begin{tikzpicture}
\node (A) at (0,0) {$\beta_{G',K} : H^1(K,(G')_{\op{ad}})$};
\node (B) at (4,0) {$H^2(K,Z_{G'})$};

\draw[->,>=latex] (A) to (B);
\end{tikzpicture}$$
has trivial kernel.

\end{enumerate}
\end{prop}

\begin{proof}
 The coradical $C_G$ is equal to $Z_G/Z_{G'}$. As it is special, we know that the morphism 
$$ \begin{tikzpicture}
\node (A) at (0,0) {$H^1(K,Z_{G'})$};
\node (B) at (3,0) {$H^1(K,Z_{G})$};

\draw[->,>=latex] (A) to (B);
\end{tikzpicture}$$
is surjective. Therefore, by the commutative diagram : 
  $$ \begin{tikzpicture}
\node (A) at (0,1) {$G_{\op{ad}}(K)$};
\node (B) at (4,2) {$H^1(K,Z_{G'})$};
\node (C) at (4,0) {$H^1(K,Z_G)$};

\node (E) at (1.8,0.3) {$\alpha_{G,K}$};
\node (F) at (1.8,1.7) {$\alpha_{G',K}$};
\draw[->,>=latex] (A) to (B);
\draw[->,>=latex] (A) to (C);
\draw[->,>=latex] (B) to (C);

\end{tikzpicture}$$
we see that {\it(1)} is equivalent to $\alpha_{G,K}$ being surjective.
Similarly, because the coradical $C_G$ is special, we know that the morphism 
$$ \begin{tikzpicture}
\node (A) at (0,0) {$H^2(K,Z_{G'})$};
\node (B) at (3,0) {$H^2(K,Z_{G})$};

\draw[->,>=latex] (A) to (B);
\end{tikzpicture}$$
induced by the inclusion is injective. Therefore, by the commutative diagram :
$$ \begin{tikzpicture}
\node (A) at (0,1) {$H^1(K,G_{\op{ad}})$};
\node (B) at (4,2) {$H^2(K,Z_{G'})$};
\node (C) at (4,0) {$H^2(K,Z_G)$};

\node (E) at (2,1.85) {$\beta_{G',K}$};
\node (F) at (2,0.25) {$\beta_{G,K}$};
%\node (G) at (4.3,1) {$\delta_K$};
\draw[->,>=latex] (A) to (B);
\draw[->,>=latex] (A) to (C);
\draw[->,>=latex] (B) to (C);

\end{tikzpicture}$$
we see that {\it(2)} is equivalent to the fact that $\beta_{G,K}$ has a trivial kernel. We can conclude that $G$ is special by Proposition \ref{immediate}.
 
\end{proof}

\section{The derived subgroup and the coradical of a special reductive group}\label{derivedcoradical}
In this section we will determine which algebraic groups can arise as derived subgroups of a special reductive group and which can arise as coradicals
respectively in Proposition \ref{derived} and \ref{coradical}
below. 

\subsection{A lemma on hermitian forms}
In order to lighten the proof of Proposition \ref{derived}, we start by proving Lemma \ref{lemme} below about hermitian forms. 
We refer the reader to
\cite[§$4$]{Inv} for the definition of hermitian forms on a right module over an algebra $D$ 
equipped with an involution.

\medskip

Let $D$ be a division algebra, $k$ a subfield of its center and $\tau$ an involution of $D$. Let $n$ be an integer and
$t_1,\ldots,t_n$ be algebraically independent variables over $k$. We denote by $K$ be the field of fractions $k(t_1,\ldots,t_n)$. 
We fix an integer $m$, a collection of scalars $\alpha_1,\ldots,\alpha_n$ in $k^*$  
and, for every index $i$ between $1$ and $m$, we fix an element $a_i=(a_{i,1},\ldots,a_{i,n})$ of $\mathbb{Z}^n$.
\begin{lemme}\label{lemme}
Suppose that the images of the $a_i$s in $(\mathbb{Z}/2\mathbb{Z})^n$ are all different.
Then, the hermitian form :
$$h(x,y)=\sum_{i=1}^m \alpha_i t_1^{a_{i,1}}\ldots t_n^{a_{i,n}} \tau(x_i)y_i$$
is anistropic on $(D\otimes_k K)^m$.
 
\end{lemme}
\begin{proof}
The proof goes along the same line as \cite[p.$111$]{Pfist}. Suppose that there exists an isotropic vector $x$. By clearing the denominator we can further assume that 
all the coordinates $x_i$ of $x$ belong to $D\otimes_k k[t_1,\ldots,t_n]$. Now, as a consequence of our assumption, we see that the leading monomials 
of the Laurent polynomials $$\alpha_i t_1^{a_{i,1}}\ldots t_n^{a_{i,n}} \tau(x_i)x_i$$ with respect to the lexicographic order are all different when $i$ ranges from 
$1$ to $m$. Therefore they cannot cancel. 
\end{proof}

\subsection{The derived subgroup of a special reductive group}

We will use \cite[§$26$]{Inv} as a basic reference
for the classification of algebraic groups over non-algebraically
closed fields. We will adopt the notations of \cite{Inv} throughout.
\begin{prop}\label{derived}
Let $G$ be a special reductive algebraic group over $k$. The derived subgroup of $G$ is isomorphic to  
$$R_{K_1|k}(G_1)\times R_{K_2|k}(G_2)\times\cdots\times R_{K_r|k}(G_r)$$
where, for each $i$, the extension $K_i$ of $k$ is finite and separable and the group $G_i$ is isomorphic over $K_i$ to either $\op{SL}_1(A_i)$, where 
$A_i$ is a central simple algebra over $K_i$, or $\op{Sp}_{2n_i}$ for some integer $n_i$.
\end{prop}

\begin{proof}
By Theorem \ref{algclos}, the group $G'_{\bar{k}}$, where $\bar{k}$ is an algebraic closure of $k$, is a semisimple simply connected group whose simple components are of type A and C. Therefore, by \cite[Theorem $26.8$]{Inv}, 
the group $G'$ is isomorphic to a direct product
$$R_{K_1|k}(G_1)\times R_{K_2|k_2}(G_2)\times\cdots\times R_{K_r|k}(G_r)$$
where, for each index $i$, the extension $K_i$ of $k$ is finite and separable and the group $G_i$ is an absolutely simple simply connected group over $K_i$ of type A or C.
For each index $i$, $G_i$ is a direct factor of the derived subgroup of the special reductive group $G_{K_i}$. By Proposition \ref{plusdinfos} we get that the map
$\beta_{G_{K_i},K}$ has trivial kernel for every field extension $K$ of $K_i$, which readily implies that the map
$\beta_{G_i,K}$ has trivial kernel as well. This forces $G_i$ to be of inner type A or split of type C, by Lemma \ref{outerA} below, completing
the proof of the proposition.

\begin{lemme}\label{outerA}
Let $G$ be an absolutely simple simply connected group of type A or C over the field $k$. If, for every field extension $K$ of $k$, the map $\beta_{G,K}$ has a trivial kernel, then 
$G$ is either of inner type A or split of type C.
\end{lemme}

{\it Proof of Lemma \ref{outerA}}. Suppose first that $G$ is of outer type A. We will prove that the kernel of $\beta_{G,K}$ contains at least two elements for 
some field extension 
$K$ of $k$. Observe that to prove this property we can replace $G$ by $G_M$ for some scalar extension $M$ of $k$.
By \cite[§$26$]{Inv}, $G$ is isomorphic to $\op{SU}(A,\si)$, where $A$ is a central simple 
algebra of degree $n$ - at least $3$, otherwise $\op{SU}(A,\si)$ is of inner type - over a quadratic separable extension $L$ of $k$ equipped with an involution $\si$ of 
the second kind.

\medskip

We will now reduce to the case where $A$ is split over $L$. 
To this aim, we denote by $Y$ the Severi-Brauer variety of $A$, by $X$ the Weil scalar restriction of $Y$ from $L$ to $k$, and by $K$ be the function field of $X$. 
As $X$ is geometrically integral, the field $k$ is algebraically closed in $K$, and consequently $K\otimes_k L$ is a field. Moreover the set $Y(K\otimes_k L)$ is not 
empty, as it is equal to $X(K)$. This implies that the field extension $K\otimes_k L$ of $L$ is a splitting field for $A$. 
Now, we observe that the group $G_K$ is isomorphic to $\op{SU}(K\otimes_k A,\si_K)$, where $K\otimes_k A$ is a split central simple algebra over $K\otimes_k L$
equipped with an 
involution $\si_K$ of the second kind. It is thus of outer type A and satisfies moreover the property that for every field extension $M$ of $K$, 
the map $\beta_{G_K,M}$ has trivial kernel.
 Therefore, by replacing $k$ by $K$ and $G$ by $G_K$, we are reduced to the case where the central simple algebra $A$ is split
over $L$.

\medskip

Then $A$ is isomorphic to $\op{End}_L(L^n)$ for some integer $n$ greater or equal to three, and the involution $\si$ is adjoint to 
a nonsingular hermitian form $h$ on $L^n$, by \cite[§$26$]{Inv}. The group $G$ is therefore isomorphic to $\op{SU}_L(n,h)$. 
Its center is the group 
$\mu_{n[L]}$, the kernel of the norm map :
$$ \begin{tikzpicture}
\node (A) at (0,0) {$N_{L|k} : R_{L|k}(\mu_{n,L})$};
\node (B) at (3,0) {$\mu_{n,k}$};

\draw[->,>=latex] (A) to (B);
\end{tikzpicture}$$
Let $K$ be the field $k(t_1,\ldots,t_{n-1})$ where the $t_i$s are algebraically independent variables over $k$.

\medskip

We claim that the  
kernel of $\beta_{G,K}$ contains at least two elements. We have an exact sequence of pointed sets :
$$ \begin{tikzpicture}
\node (B) at (3,0) {$H^1(K,\mu_{n[L]})$};
\node (C) at (6.1,0) {$H^1(K,G)$};
\node (D) at (9.2,0) {$H^1(K,G_{\op{ad}})$};
\node (E) at (12.8,0) {$H^2(K,\mu_{n[L]})$};
\node (G) at (11,0.25) {$\beta_{G,K}$};

\draw[->,>=latex] (B) to (C);
\draw[->,>=latex] (C) to (D);
\draw[->,>=latex] (D) to (E);

\end{tikzpicture}$$
in the fppf-cohomology. As $\mu_{n[L]}$ is abelian and central in $G$, there is a natural action of $H^1(K,\mu_{n[L]})$ on $H^1(K,G)$, and the set of 
orbits for this action is precisely the kernel of $\beta_{G,K}$.
By \cite[Example $29.19$]{Inv}, the set $H^1(K,G)$ is in natural correspondence with 
the set of isometry classes of nonsingular hermitian forms on the vector space 
$(K\otimes_k L)^n$ 
with the same discriminant $\alpha$ as $h$. Moreover, by \cite[Proposition $30.13$]{Inv}, the group $H^1(K,\mu_{n[L]})$ is the quotient of 
$$\{(x,y)\in K^*\times (K\otimes_k L)^*,\quad x^n=N_{K\otimes_k L|K}(y)\}$$
by the subgroup 
$$\{(N_{K\otimes_k L|K}(z),z^n),\quad z\in (K\otimes_k L)^*\}.$$
Strictly speaking, the description above is given in \cite[Proposition $30.13$]{Inv} only when $n$ is not divisible by the characteristic of the base field $k$. 
This comes from the fact that the cohomology considered there is the Galois cohomology. The same proof leads to the description in the fppf-cohomology, with no restriction 
on the integer $n$.
It is then easy to prove that the action of the class $[(x,y)]$ on the isometry class $[h']$ of the hermitian form $h'$ is given as follows :  
$$[(x,y)]\cdot [h'] =[x h'].$$
We will now prove that the set $H^1(K,G)$ contains the isometry class of an isotropic form and an anisotropic form. As these two classes cannot be in the 
same orbit under the action of $H^1(K,\mu_{n[L]})$, this proves the claim above.

\medskip

First, as $n$ is greater than $2$, $H^1(K,G)$ contains the isometry class  
of an isotropic hermitian form, namely the one with matrix 
$$\op{diag}(\left[
\begin{array}{cc}
0 & 1  \\
1 & 0 
\end{array}
\right], 1, \cdots,1,-\alpha).$$
Moreover, by Lemma \ref{lemme} above, the hermitian form :
$$h'(x_1,\ldots ,x_n)=t_1 \si(x_1)x_1+\cdots +t_{n-1} \si(x_{n-1})x_{n-1}+\alpha t_1\ldots t_{n-1} \si(x_n)x_n$$
which has discriminant $\alpha$, is anistropic over the field $K\otimes_k L$. 

\medskip

 Suppose now that $G$ is of type C and not split. Again here, we want to see that the kernel of $\beta_{G,K}$ contains at least two elements for some 
field extension $K$ of $k$.
By \cite[§$26$]{Inv}, $G$ is isomorphic to $\op{Sp}(A,\si)$, where $A$ is a nonsplit central simple 
algebra of degree $2n$ - at least $4$, otherwise $\op{Sp}(A,\si)$ is of type A - over $k$ equipped with an involution $\si$ of 
symplectic type. The center of $G$ is isomorphic to $\mu_2$. Let $K$ be a field extension of $k$. By \cite[$(29.22)$]{Inv} the kernel of $\beta_{G,K}$ 
is in bijection 
with the conjugacy classes of involutions of symplectic type on $A_{K}$. 

\medskip
If $A$ is a division algebra, then by \cite[Theorem $3.1$]{Lot}, there are more than one conjugacy classes of involutions of symplectic type on $A_{K}$.
From now on, we suppose that $A$ is not a division algebra. Let $D$ be the division algebra Brauer equivalent 
to $A$. It is not $k$, as $A$ is nonsplit. Therefore, $D$ carries an involution $\tau$ of symplectic type, by
\cite[Theorem $3.1$]{Inv} and \cite[Corollary $2.8$]{Inv}, and, by Wedderburn's theorem \cite[Theorem $(1.1)$]{Inv}, 
$A$ is isomorphic to $M_n(D)$ for some integer $n$ greater or equal to $2$.

\medskip

Let $K$ be the field of fractions $k(t_1,\cdots,t_n)$ on $n$ indeterminates. 
The algebra $D_K$ is a division algebra, and is therefore the division algebra 
Brauer equivalent to $A_K$. Let $M$ be a simple right $A_K$-module, isomorphic to $D_K^n$ - thought of as column vectors. We will make use of the correspondence between
involutions of symplectic type on $A_K$ and hermitian forms on $M$, as explained in \cite[Theorem $(4.2)$]{Inv}. We refer the reader 
to \cite[§$4$]{Inv} for the notion of singular hermitian form and alternating hermitian form in characteristic $2$.  
We define two hermitian forms $h$ and $h'$ on $M$ in the following way :
$$h(x,y)=-\tau(x_1)y_1+\sum_{i=2}^{s}\tau(x_i)y_i$$
and 
$$h'(x,y)=\sum_{i=1}^{s}t_i\tau(x_i)y_i$$ 
These forms are easily seen to be nonsingular.
Furthermore, if the characteristic of $k$ is $2$, then $h$ and $h'$ are alternating. Indeed, 
by \cite[Proposition $2.6$]{Inv}, as $\tau$ is an involution of symplectic type on $D_K$ we know that $K$ is contained in 
$\op{Symd}(D_K,\tau)$, which is enough to prove that for every $x$ in $M$, the elements $h(x,x)$ and $h'(x,x)$ both belong to $\op{Symd}(D_K,\tau)$. 

\medskip

By \cite[Theorem $(4.2)$]{Inv}, the hermitian forms $h$ and $h'$ give rise to two involutions $\tau_h$ and $\tau_{h'}$ on $A_K$ which are both of symplectic type.
If these two involutions were conjugate, then it would exist an element $u$ of $\op{GL}_n(D_K)$ such that 
the hermitian forms $h'$ and the hermitian form : 
$$ \begin{tikzpicture}
\node (A) at (0,0) {$M\times M$};
\node (B) at (4,0) {$D\quad (x,y)\mapsto h(u(x),u(y))$};

\draw[->,>=latex] (A) to (B);
\end{tikzpicture}$$
are proportional by a factor in $K^*$. But $h$ is isotropic, for instance, $(1,1,0,\cdots,0)$ is an isotropic element, and $h$ is not by Lemma \ref{lemme}. 
This provides a contradiction, proving that the involutions $\tau_h$ and $\tau_{h'}$ are not conjugate.

\end{proof}
Observe that every group admitting a direct factor decomposition as in Proposition \ref{derived} occurs as the derived subgroup of a special reductive group.
Indeed, a semi-simple group 
$$R_{K_1|k}(G_1)\times R_{K_2|k}(G_2)\times\cdots\times R_{K_r|k}(G_r)$$
where, for each $i$, the extension $K_i$ of $k$ is finite and separable and the group $G_i$ is isomorphic over $K_i$ to either $\op{SL}_1(A_i)$, where 
$A_i$ is a central simple algebra over $K_i$, or $\op{Sp}_{2n_i}$ for some integer $n_i$, is the derived subgroup of the special reductive group 
$$R_{K_1|k}(H_1)\times R_{K_2|k}(H_2)\times\cdots\times R_{K_r|k}(H_r)$$
where, for each index $i$, $H_i$ is equal to $\op{GL}_1(A_i)$ if $G_i$ is isomorphic to $\op{SL}_1(A_i)$, and $H_i$ is equal to $G_i$ otherwise.

\subsection{The coradical of a special reductive group}

We prove now that the coradical of a special reductive group is a special torus. The classification of special tori, due to Colliot-Thélène, will be 
recalled in Section \ref{tori} below.
\begin{prop}\label{coradical}
Let $G$ be a special reductive algebraic group defined over $k$. 
The coradical $C_G$ of $G$ is a special torus. 
 
\end{prop}
\begin{proof}
We say that a reductive algebraic group $G$ defined over a field - which is not necessarily $k$ -  satisfies property $(P)$ if, for every field extension $K$ of 
the field of definition of $G$
and every non trivial element $x$ in $H^2(K,Z_G)$, there exists a field extension $L$ of $K$ such that $x_L$ is not trivial in $H^2(L,Z_G)$ and belongs to the
image of $\beta_{G,L}$.
We will prove as a consequence of Lemma \ref{stableby} below that the derived subgroup $G'$ of $G$ - and more generally any group admitting a direct factor decomposition as in Proposition 
\ref{derived} - satisfies 
property $(P)$. 

\medskip

Before proving this fact let us show how it implies the proposition. 
The coradical $C_G$ of $G$ is equal to the quotient $Z_G/Z_{G'}$. Suppose that $C_G$ is not special. There exists a field extension $K$ of $k$ and a nontrivial 
element $x$ in $H^1(K,C_G)$. By Proposition \ref{plusdinfos}, the image of $x$ in $H^2(K,Z_{G'})$ - still denoted $x$ - is nonzero, and is mapped to zero 
in $H^2(K,Z_{G})$. 
As the group $G'$ satisfies property $(P)$, we can even assume, after possibly extending scalars, that there exists $y$ in $H^1(K,(G')_{\op{ad}})$ such that 
$x=\beta_{G',K}(y)$. We get that $y$ is not trivial and is in the kernel of $\beta_{G,K}$, a contradiction to Proposition \ref{immediate}.

\medskip

Now, the fact that any group admitting a direct factor decomposition as in Proposition 
\ref{derived} satisfies 
property $(P)$ is a direct consequence of Lemma \ref{stableby} below. In this lemma, we say that that a set of reductive algebraic groups 
- not necessarily defined over the same base field - is stable under Weil scalar restriction if for every field $k$, for every finite separable field extension 
$K$ of $k$ and for every reductive algebraic group $G$ defined over $K$ in the set the group $R_{K|k}(G)$ belongs to the set as well.

\begin{lemme}\label{stableby}
The set of reductive algebraic groups satisfying property $(P)$ :
\begin{enumerate}[(1)]
 \item contains $\op{SL}_1(A)$ for every central simple algebra $A$ over a field $k$.
\item contains $\op{Sp}_{2n}$, for every integer $n$ and every field $k$.
 \item is stable under finite direct products.
 \item is stable under Weil scalar restrictions. 
 \end{enumerate}
 
\end{lemme}
{\it Proof of Lemma \ref{stableby}}. Let $k$ be a field and $A$ a central simple algebra over $k$. The center of $\op{SL}_1(A)$ is $\mu_n$, where $n$ is the degree of $A$. 
Let $K$ be a field extension of $k$ and $x$ a nontrivial element of $H^2(K,\mu_n)$. The element $x$ is the Brauer class 
of a central simple algebra $B$ over $K$ of period $d$ dividing $n$, $d$ being greater than $1$. By the Schofield-Van den Bergh index reduction formula 
\cite[Theorem 2.5]{SchoVan}, 
there exists a field extension $L$ of $K$ 
such that 
$B_L$ is a central simple algebra over $L$ of index $d$. This proves that the class $x_L$ is not trivial in $H^2(L,\mu_n)$ because $d$ is not $1$, and belongs to 
the image of $\beta_{\op{SL}_1(A),L}$, this image being precisely the classes of index dividing $n$. We have proved {\it(1)}.

\medskip

The proof of {\it(2)} is similar. Let $k$ be a field and $n$ an integer. The center of $\op{Sp}_{2n}$ is $\mu_2$. 
Let $K$ be a field extension of $k$ and $x$ a nontrivial element of $H^2(K,\mu_2)$. 
The element $x$ is the Brauer class 
of a central simple algebra $B$ over $K$ of period $2$. Applying the index reduction formula once again, there exists a field extension $L$ of $K$ 
such that 
$B_L$ is a central simple algebra over $L$ of index $2$. This proves that the class $x_L$ is not trivial in $H^2(L,\mu_2)$, and belongs to 
the image of $\beta_{\op{Sp}_{2n},L}$, this image being precisely the classes of index dividing $2n$.

\medskip

Let $k$ be a field. We will now prove that if $G_1$ and $G_2$ are reductive algebraic groups both satisfying property $(P)$, then the direct product 
$G_1\times G_2$ satisfies property $(P)$ as well. Let $K$ be a field extension of $k$ and $x$ be a nontrivial element of 
$$H^2(K,Z_{G_1\times G_2})=H^2(K,Z_{G_1})\times H^2(K,Z_{G_2}).$$
We write $x=(x_1,x_2)$. As $G_1$ satisfies property $(P)$, there exists a field extension $L_1$ of $K$ such that $(x_1)_{L_1}$ is not trivial and 
belongs to the image of $\beta_{G_1,L_1}$. If $(x_2)_{L_1}$ is trivial then we are done. Otherwise, as $G_2$ satisfies property $(P)$, there exists a field extension $L_2$ 
of $L_1$ such that $(x_2)_{L_2}$ is not trivial and 
belongs to the image of $\beta_{G_2,L_2}$. As $(x_1)_{L_2}$ belongs to the image of $\beta_{G_1,L_2}$, we see that $x_{L_2}$ is not trivial and belongs to the 
image of $\beta_{G_1\times G_2,L_2}$. This completes the proof of {\it(3)}.

\medskip 

Let $k$ be a field, $M$ a finite separable field extension of $k$, and let $G$ be reductive algebraic group over $M$ which satisfies property $(P)$.
We will now prove that the group $R_{M|k}(G)$ satisfies property $(P)$ as well.
We denote by $d$ the degree of the field extension $M$ of $k$. Let $K$ be a field extension of $k$.
We can write 
$$K\otimes_k M=K_1\times \cdots \times K_s$$
where the $K_i$s are finite separable extensions of $K$ and $M$.  
Let $x$ a nontrivial element 
in 
$$H^2(K,Z_{R_{M|k}(G)})=H^2(K_1,Z_G)\times\cdots\times H^1(K_s,Z_G).$$
We write $x=(x_1,\ldots,x_s)$, and we define $d_x$ to be the sum of the degrees of the $K_i$s over $k$ such that $x_i$ belongs to the image of 
$\beta_{G,K_i}$. The integer $d_x$ is obviously less than or equal to $d$. We prove the desired conclusion by a decreasing induction on $d_x$, 
the case where $d_x$ is equal to 
$d$ being obvious. Suppose that $d_x$ is strictly less than $d$. After permuting the $K_i$s, we can assume for example that $x_1$ is not in the image of 
$\beta_{G,K_1}$. In particular, $x_1$ is not trivial. As $G$ satisfies property $(P)$, there exists a field extension $L_1$ of $K_1$ such that 
$(x_1)_{L_1}$ is not trivial and belongs to the image of $\beta_{G,L_1}$. 
One then easily proves that $x_{L_1}$ is not trivial and $d_{x_{L_1}}$ 
is strictly greater than $d_x$. By the induction hypothesis there is a field extension $L$ of $L_1$ such that $x_L$ is not trivial and belongs to the 
image of $\beta_{G,L}$, completing the proof of {\it(4)}.

\end{proof}
We make now the observation that the radical of a special reductive group does not need to be special. Suppose that $K$ is a separable quadratic extension of $k$.
Recall that the torus $R^1_{K|k}(\mathbb{G}_m)$ is defined as the kernel of the norm map from $R_{K|k}(\mathbb{G}_m)$ to $\mathbb{G}_m$. We denote by $R$ 
the direct product $R^1_{K|k}(\mathbb{G}_m)\times \mathbb{G}_m$.
There is an exact sequence of algebraic tori :
$$
\begin{tikzpicture}
\node (A) at (0,0) {$1$};
\node (B) at (1.5,0) {$\mu_2$};
\node (C) at (3.5,0) {$R$};
\node (D) at (6.2,0) {$R_{K|k}(\mathbb{G}_m)$};
\node (E) at (8.2,0) {$1$};
\node (F) at (2.6,0.2) {$\varphi$};
\draw[->,>=latex] (A) to (B);
\draw[->,>=latex] (B) to (C);
\draw[->,>=latex] (C) to (D);
\draw[->,>=latex] (D) to (E);
\end{tikzpicture}
$$
corresponding to the following exact sequence of $\Gamma$-modules, where $\Gamma$ is the Galois group of $K$ over $k$ : 
$$
\begin{tikzpicture}
\node (A) at (0,0) {$0$};
\node (B) at (1.5,0) {$\mathbb{Z}^2$};
\node (C) at (4.2,0) {$\mathbb{Z}^2$};
\node (D) at (7,0) {$\mathbb{Z}/2\mathbb{Z}$};
\node (E) at (8.6,0) {$0$};
%\node (F) at (2.5,0.2) {$f$};
%\node (G) at (4.6,0.2) {$g$};
\node (H) at (1.2,-0.7) {$(x,y)$};
\node (I) at (4.3,-0.7) {$(x-y,x+y)$};
\node (J) at (4.2,-1.4) {$(x,y)$};
\node (K) at (6.8,-1.4) {$[x+y]$};
\draw[->,>=latex] (A) to (B);
\draw[->,>=latex] (B) to (C);
\draw[->,>=latex] (C) to (D);
\draw[->,>=latex] (D) to (E);
\draw[->,>=latex] (H) to (I);
\draw[->,>=latex] (J) to (K);
%\draw[->,>=latex] (D) to (E);
\end{tikzpicture}
$$
Here the nontrivial element of $\Gamma$ acts on $\mathbb{Z}^2$ on the left by permuting the coordinates, on $\mathbb{Z}^2$ in the center by multiplying the 
first coordinate by $-1$ and the second by $1$, and on $\mathbb{Z}/2\mathbb{Z}$ as the identity. We define $G$ to be the quotient : 
$$(\op{SL}_2\times R)/\mu_2,$$  
where $\mu_2$ is embedded diagonally in $\op{SL}_2$ and in $R$ by using the morphism $\varphi$ above. It is readily seen that the derived group of $G$ is 
$\op{SL}_2$ and its coradical 
is $R_{K|k}(\mathbb{G}_m)$. An easy argument then shows that $G$ is special, see for instance Proposition \ref{qsplit} below. However, the radical of $G$ is 
equal to $R=R^1_{K|k}(\mathbb{G}_m)\times \mathbb{G}_m$ and is therefore not special, as it can be seen directly or from the classification of special tori 
recalled in Theorem \ref{specialtori} below.

\section{Classification results}\label{classification}
 We start by classifying special reductive groups over the field $k$ in Theorem \ref{classif2} below. This classification is
 obtained as a straightforward consequence of the results from 
Sections \ref{preliminary} and \ref{derivedcoradical}. However, conditions {\it(1)} and {\it(2)} in Theorem \ref{classif2} 
are very explicit, unlike condition {\it(3)}. Under the additional assumption that the group $G$ is semisimple, reductive of inner type or quasisplit, 
we will make condition {\it(3)} explicit as well, providing an explicit classification in these cases.
\begin{thm}\label{classif2}
Let $G$ be a reductive algebraic group over $k$. Then $G$ is special if and only if the following three conditions hold :
\begin{enumerate}[(1)]
 \item The derived subgroup of $G$ is isomorphic to  
$$R_{K_1|k}(G_1)\times R_{K_2|k}(G_2)\times\cdots\times R_{K_r|k}(G_r)$$
where, for each $i$, the extension $K_i$ of $k$ is finite and separable and the group $G_i$ is isomorphic over $K_i$ to either $\op{SL}_1(A_i)$, where 
$A_i$ is a central simple algebra over $K_i$, or $\op{Sp}_{2n_i}$ for some integer $n_i$.
 \item The coradical $C_G$ of $G$ is a special torus.
 \item For every field extension $K$ of $k$, we have 
 $$\op{Im}(\alpha_{G',K})+\op{Ker}(H^1(K,Z_{G'})\rw H^1(K,Z_{G}))=H^1(K,Z_{G'}).$$
\end{enumerate}
\end{thm}

\begin{proof}
 If $G$ is special, then {\it(1)} is satisfied by Proposition \ref{derived}, {\it(2)} by Proposition \ref{coradical} and {\it(3)} by 
 Proposition \ref{classif1}. Suppose now that $G$ satisfies the three conditions. By {\it(1)} it is easily seen that for every field extension $K$ of $k$, 
 the map $\beta_{G',K}$ has trivial kernel. Together with {\it(2)} and {\it(3)}, it implies that $G$ is special, by Proposition \ref{classif1}.
\end{proof}

\subsection{The classification of special semisimple groups}

We provide now the classification of special semisimple groups over the field $k$.
\begin{prop}\label{semisimple}
Let $G$ be a semisimple algebraic group over $k$. Then $G$ is special if and only if it is isomorphic to 
$$R_{K_1|k}(G_1)\times R_{K_2|k_2}(G_2)\times\cdots\times R_{K_r|k}(G_r)$$
where, for each $i$, the extension $K_i$ of $k$ is finite and separable and the group $G_i$ is isomorphic over $K_i$ to $\op{SL}_{n_i}$ or $\op{Sp}_{2n_i}$ 
for some integer $n_i$.
\end{prop}
\begin{proof}
The ``if part'' of the proposition follows directly from Schapiro's lemma and the fact that the split groups $\op{SL}_{n}$ and $\op{Sp}_{2n}$ are special
for every integer $n$. For the ``only if part'', we use Proposition \ref{derived}. As $G$ is its own derived subgroup, we find that $G$ is isomorphic to  
$$R_{K_1|k}(G_1)\times R_{K_2|k}(G_2)\times\cdots\times R_{K_r|k}(G_r)$$
where, for each $i$, the extension $K_i$ of $k$ is finite and separable and the group $G_i$ is isomorphic over $K_i$ to either $\op{SL}_1(A_i)$, where 
$A_i$ is a central simple algebra over $K_i$, or $\op{Sp}_{2n_i}$ for some integer $n_i$. Now, for every index $i$, $G_i$ is a direct factor of $G_{K_i}$, and, as such, 
is a special group. If $G_i$ is isomorphic over $K_i$ to $\op{SL}_1(A_i)$ then Lemma \ref{nrd} below shows that $A_i$ is split, completing the proof of 
the proposition.

\begin{lemme}\label{nrd}
Let $A$ be a central simple algebra over the field $k$. If $\op{SL}_1(A)$ is a special group, then $A$ is split.
\end{lemme}
This result is part of the 
folklore, see for example \cite[Chapter $2$, Exercise $6$]{CSA}. We sketch a proof for the convenience of the reader.
Let $k((t))$ be the field of formal Laurent series.
By \cite[Corollary $29.4$]{Inv} we know that the set $H^1(k((t)),\op{SL}_1(A))$ is naturally 
identified with the cokernel of the reduced norm :
$$ \begin{tikzpicture}
\node (B) at (0.5,0) {$\op{Nrd} : (k((t))\otimes_k A)^*$};
\node (C) at (4,0) {$k((t))^*$};
\draw[->,>=latex] (B) to (C);
\end{tikzpicture}$$
We claim that the class $[t]$ of $t$ in $H^1(k((t)),\op{SL}_1(A))$ is not trivial if $A$ is not split, proving that $\op{SL}_1(A)$ is not special in that case.

\medskip

To see this, we will prove below that the image of the composite :
$$ \begin{tikzpicture}
\node (B) at (0,0) {$v\circ \op{Nrd} : (k((t))\otimes_k A)^*$};
\node (C) at (3.8,0) {$k((t))^*$};
\node (D) at (5.6,0) {$\mathbb{Z}$};
%\node (F) at (5.8,0.3) {$v$};
%\node (G) at (3,0.3) {$\op{Nrd}$};
\draw[->,>=latex] (B) to (C);
\draw[->,>=latex] (C) to (D);
\end{tikzpicture}$$ 
where $v$ is the valuation given by $t$, is the ideal spanned by the index $\op{ind}(A)$ of $A$.
Let us first show how it implies the result. If $A$ is not split, 
 then the index of $A$ is not $1$, and we see that $t$, whose valuation is $1$, is not in the image of $\op{Nrd}$, proving that the class $[t]$ in 
 $H^1(k((t)),\op{SL}_1(A))$ is not trivial.

 \medskip

 We now prove the result above.
Let $D$ be the division algebra over $k$ which is Brauer equivalent to $A$. 
 Observe that the valuation $v$ extends to $k((t))\otimes_k D$, the valuation of 
$$d_r t^r + d_{r+1}t^{r+1} +\cdots$$
being $r$ if $d_r$ is not zero. This implies actually that the 
$k((t))$-algebra $k((t))\otimes_k D$ is a division algebra, and is thus the division algebra Brauer-equivalent to 
$k((t))\otimes_k A$. By \cite[Corollary $2.8.10$]{CSA}, the image in $k((t))^*$ of the reduced norms from $(k((t))\otimes_k A)^*$ and $(k((t))\otimes_k D)^*$
are the same. Therefore, in order to prove the result above, we can replace $A$ by $D$. By extending the scalars to $k_s((t))$ - which is contained in a separable closure of 
$k((t))$ -
where the reduced norm becomes the determinant, we see that the valuation of the reduced norm of 
$$d_r t^r + d_{r+1}t^{r+1} +\cdots$$
 where $d_r$ is not zero, is $r\op{dim}_k D$, which is equal to $r\op{ind}(A)$. This completes the proof of the result above. 

\end{proof}

\subsection{The classification of special reductive groups of inner type}
The {\bf split form} of a reductive algebraic group $G$ defined over $k$ is the unique Chevalley group $G_{\op{split}}$ over $k$ which is isomorphic to 
$G$ over the algebraic closure $\bar{k}$ of $k$. The existence and uniqueness of the split form is guaranteed by Chevalley's classification of 
split reductive groups, see for instance \cite{Sp}, and the fact that every reductive group is split over an algebraically closed field. 

\medskip

A reductive algebraic group $G$ is called of {\bf inner type} if it is an inner form of its split form, that is, 
if it is obtained by twisting $G_{\op{split}}$ by a cocycle with values in the group of inner automorphisms of $G_{\op{split}}$, see for instance \cite[§$31$]{Inv}.
If $G$ is a reductive group of inner type, then :
$$Z_G=Z_{G_{\op{split}}},\quad Z_{G'}=Z_{(G_{\op{split}})'} \text{  and  } R_G=R_{G_{\op{split}}}.$$
Consequently, we see that $Z_G$ and $Z_{G'}$ are split diagonalizable groups and $R_G$ is a split torus. 
We provide now the classification of special reductive algebraic groups which are of inner type.

\begin{prop}\label{inner}
Let $G$ be a reductive algebraic group over $k$ of inner type. The intersection $R'_G$ of $R_G$ with $Z_{G'}$ is a finite split diagonalizable group. We fix an isomorphism: 
$$R'_G \simeq \mu_{m_1}\times\cdots\times\mu_{m_q}$$
for some integers $m_j$.
Then $G$ is special if and only if the following two conditions are satisfied :
\begin{enumerate}[(1)]
 \item The derived subgroup $G'$ of $G$ is isomorphic to a direct product :
 $$G_1\times \cdots \times G_s \times \cdots\times G_r \quad \quad \quad (*)$$
 where, for each index $i$ from $1$ to $s$, the group $G_i$ is equal to $\op{SL}_1(A_i)$, with $A_i$ a nonsplit central simple algebra of degree $n_i$ and index 
 $d_i$ over $k$, 
 and, for $i$ from $s+1$ to $r$, the group $G_i$ is equal to either $\op{SL}_{n_i}$ or $\op{Sp}_{2n_i}$ for some integer $n_i$.

\item  The projection onto the first $s$ factors 
 in the direct product decomposition $(*)$ leads to a morphism :
 $$R'_G\simeq \mu_{m_1}\times\cdots\times\mu_{m_q}\rw Z_{G_1\times\cdots\times G_s}=\mu_{n_1}\times\cdots\times\mu_{n_s}$$
$$(x_1,\ldots,x_q)\mapsto (x_1^{a_{1,1}}\cdots x_q^{a_{1,q}},\ldots,x_1^{a_{s,1}}\cdots x_q^{a_{s,q}})$$
for some integers $a_{i,j}$s. We set $b_{i,j}=\frac{a_{i,j}n_i}{m_j}$.
Then the rows of the following matrix :
$$\left[
\begin{array}{cccccccccccc}
d_1 & 0 & \ldots & 0 & b_{1,1} & \ldots & b_{1,q} \\
0 & d_2 & \ldots & 0 & b_{2,1} & \ldots & b_{2,q}\\
\ldots & \ldots & \ldots & \ldots & \ldots & \ldots & \ldots\\
0 & 0 & \ldots & d_s & b_{s,1} & \ldots & b_{s,q}\\
\end{array}
\right]$$
span a saturated sublattice of $\mathbb{Z}^{s+q}$. 
 
\end{enumerate}

\end{prop}
\begin{proof}
 
First we prove that if $G$ is special then it satisfies {\it (1)}. 
As $G$ is of inner type, it is obtained by twisting the split form $G_{\op{split}}$ of $G$ by a cocycle whose class is in $H^1(k,(G_{\op{split}})_{\op{ad}})$.
As the last set is equal to $H^1(k,(G'_{\op{split}})_{\op{ad}})$, and the group $G'_{\op{split}}$ is a direct product of split absolutely simple simply connected groups -
because it is the derived subgroup of the special split reductive group $G_{\op{split}}$ -
we see that $G'$, which is obtained from $G'_{\op{split}}$ by the same twisting procedure, is a direct product of absolutely simple simply connected groups of type A and 
C. By Proposition \ref{derived}, the factors are either of inner type A or split of type C, proving that $G$ satisfies {\it(1)}.

\medskip

 We suppose now that $G$ satisfies {\it(1)}, and we claim that $G$ is special if and only if it satisfies {\it(2)}.
It is readily seen that $G$ satisfies the first assertion of Theorem \ref{classif2}, and also the second, as the coradical of $G$ is a split torus - it is the 
coradical of the split form $G_{\op{split}}$ of $G$. 
Therefore to prove the claim, it suffices to show that {\it(2)} is satisfied if and only if the third assertion of Theorem \ref{classif2} is satisfied.
In order to do this, we first identify the kernel of the following morphism :
$$H^1(K,Z_{G'})\rw H^1(K,Z_{G}).$$
There is an isomorphism from the group $Z_G$ to $R_G\times (Z_{G'}/R_G')$ such that the projection of the subgroup 
$Z_{G'}$ on the second factor 
is the natural projection from $Z_{G'}$ onto $Z_{G'}/R_G'$. This follows from the fact that the inclusion of $G'$ into $G$ provides an isomorphism 
from $G'/R_G'$ to $G/R_G$, and therefore an isomorphism from the center $Z_{G'}/R_G'$ of the first group to the center $Z_G/R_G$ of the second group, together 
with the fact that $Z_G$ is a split diagonalizable group and $R_G$ is a split torus, implying that the exact sequence 
$$ \begin{tikzpicture}

\node (A) at (-0.1,0) {$1$};
\node (B) at (1.2,0) {$R_G$};
\node (C) at (2.9,0) {$Z_G$};
\node (D) at (4.9,0) {$Z_G/R_G$};
\node (E) at (6.5,0) {$1$};

\draw[->,>=latex] (A) to (B);
\draw[->,>=latex] (B) to (C);
\draw[->,>=latex] (C) to (D);
\draw[->,>=latex] (D) to (E);

\end{tikzpicture}$$
splits. Let $K$ be a field extension of $k$. By using the isomorphism above and Hilbert's theorem 90 the set $H^1(K,Z_G)$ is identified with $H^1(K,Z_{G'}/R_G')$ 
and the morphism from $H^1(K,Z_{G'})$ to $H^1(K,Z_G)$ induced by the inclusion with the morphism:
$$ \begin{tikzpicture}
\node (B) at (0,0) {$H^1(K,Z_{G'})$};
\node (C) at (3.5,0) {$H^1(K,Z_{G'}/R_G')$};
\draw[->,>=latex] (B) to (C);
\end{tikzpicture}$$
given by the projection from $Z_{G'}$ to $Z_{G'}/R_G'$. This proves the following equality
$$\op{Ker}(H^1(K,Z_{G'})\rw H^1(K,Z_G))=\op{Im}(H^1(K,R_G')\rw H^1(K,Z_{G'})).$$ 
Observe now that the map $\alpha_{G',K}$ is the direct product of the maps $\alpha_{G_i,K}$, where $i$ ranges from $1$ to $r$.
As the group $G_i$ is special for $i$ between $s+1$ and $r$ we know by Proposition \ref{immediate} that $\alpha_{G_i,K}$ is surjective.
As a consequence, the third assertion in Theorem \ref{classif2} is satisfied by the group $G$ if and only if 
$$\op{Im}(\alpha_{G_1\times\cdots\times G_s,K})+\op{Im}(H^1(K,R_G')\rw H^1(K,Z_{G_1\times\cdots\times G_s}))=H^1(K,Z_{G_1\times\cdots\times G_s})$$
where the morphism
$$H^1(K,R_G')\rw H^1(K,Z_{G_1\times\cdots\times G_s})$$
is induced by the composite $\varphi$ of the inclusion of $R_G'$ in $Z_{G'}$ followed by the projection on $Z_{G_1\times\cdots\times G_s}$.
The morphism $\varphi$ has the following explicit description :
$$\varphi : \mu_{m_1}\times\cdots\times\mu_{m_q}\rw \mu_{n_1}\times\cdots\times\mu_{n_r}$$
$$(x_1,\ldots,x_q)\mapsto (x_1^{a_{1,1}}\cdots x_q^{a_{1,q}},\ldots,x_1^{a_{r,1}}\cdots x_q^{a_{r,q}})$$
Its corresponding morphism in fppf cohomology is given by:
$$K^*/(K^*)^{(m_1)}\times\cdots\times K^*/(K^*)^{(m_q)}\rw K^*/(K^*)^{(n_1)}\times\cdots\times K^*/(K^*)^{(n_s)}$$
$$([x_1],\ldots,[x_q])\mapsto ([x_1^{b_{1,1}}\cdots x_q^{b_{1,q}}],\ldots,[x_1^{b_{s,1}}\cdots x_q^{b_{s,q}}])$$
where $b_{i,j}=\frac{a_{i,j}n_i}{m_j}$, and $(K^*)^{(n)}$ denotes the group of $n$th power of elements of $K^*$.
Furthermore, for each index $i$ from $1$ to $s$, the map $\alpha_{G_i,K}$ :
$$ \begin{tikzpicture}
\node (B) at (0,0) {$\op{PSL}_1(A_i)(K)=(A_i)_K^*/K^*$};
\node (C) at (5.6,0) {$H^1(K,\mu_{n_i})=K^*/(K^*)^{(n_i)}$};
\draw[->,>=latex] (B) to (C);
\end{tikzpicture}$$
maps the class of an element $g$ of $(A_{i})_K^*$ to the class of its reduced norm.
Therefore, Lemma \ref{final} below completes the proof of the proposition.

\begin{lemme}\label{final}
The following conditions are equivalent :
\begin{enumerate}[(1)]
 \item the rows of the following matrix :
$$\left[
\begin{array}{cccccccccccc}
d_1 & 0 & \ldots & 0 & b_{1,1} & \ldots & b_{1,q} \\
0 & d_2 & \ldots & 0 & b_{2,1} & \ldots & b_{2,q}\\
\ldots & \ldots & \ldots & \ldots & \ldots & \ldots & \ldots\\
0 & 0 & \ldots & d_s & b_{s,1} & \ldots & b_{s,q}\\
\end{array}
\right]$$
 span a saturated sublattice of $\mathbb{Z}^{s+q}$.
 \item for every field extension $K$ of $k$, the map :
$$\gamma_K : \prod_{i=1}^s\op{Nrd}((A_i)_K^*)\times (K^*)^q\rw (K^*)^r$$
$$(y_1,\ldots,y_s,x_1,\ldots,x_q)\mapsto (x_1^{b_{1,1}}\cdots x_q^{b_{1,q}}y_1,\ldots,x_1^{b_{r,1}}\cdots x_q^{b_{r,q}}y_s)$$
is surjective.
\end{enumerate}
\end{lemme}
{\it Proof of Lemma \ref{final}}.
Suppose that {\it(1)} hold. Let $M$ be the matrix in {\it(1)}.
First, as the rows of $M$ are linearly independent, the morphism of algebraic tori :
$$\mathbb{G}_m^{q+s}\rw \mathbb{G}_m^s$$
$$(y_1,\ldots,y_s,x_1,\ldots,x_q)\mapsto (x_1^{b_{1,1}}\cdots x_q^{b_{1,q}}y_1^{d_1},\ldots,x_1^{b_{s,1}}\cdots x_q^{b_{s,q}}y_s^{d_s})$$
is surjective. Its kernel is precisely the subtorus of $\mathbb{G}_m^{s+q}$ whose character lattice is the quotient of $\mathbb{Z}^{s+q}$ by the 
rows of $M$. By assumption this kernel is therefore a split torus. By Hilbert's theorem 90, for every field extension $K$ of $k$, the map :
$$(K^*)^{s+q}\rw (K^*)^s$$
$$(y_1,\ldots,y_s,x_1,\ldots,x_q)\mapsto (x_1^{b_{1,1}}\cdots x_q^{b_{1,q}}y_1^{d_1},\ldots,x_1^{b_{s,1}}\cdots x_q^{b_{s,q}}y_s^{d_s})$$
induced on the $K$-points is surjective. As for each index $i$ between $1$ and $s$ the subgroup $\op{Nrd}((A_i)_K^*)$ of $K^*$ contains 
$(K^*)^{(d_i)}$, we see that the map $\gamma_K$ is surjective.

\medskip

Suppose now that {\it (1)} fails. There exists a primitive element 
$(c_1,\ldots,c_{s})$ of $\mathbb{Z}^{s}$
such that the element 
$$\sum_{i=1}^s c_i(0,\ldots,d_i,\ldots,0,b_{i,1},\ldots,b_{i,q})$$
is divisible, say by $d$, in $\mathbb{Z}^{s+q}$. Let $K$ be the field of Laurent series $k((t))$ and $v$ the valuation defined by $t$. 
As the element $(c_1,\ldots,c_{s})$ is primitive, 
the map 
$$(K^*)^s\rw K^*$$
$$(z_1,\ldots,z_s)\mapsto z_1^{c_1}\ldots z_s^{c_s}$$
is surjective.
We claim now that if $(z_1,\ldots,z_s)$ belongs to the image of 
$\gamma_K$ then the valuation of $z_1^{c_1}\ldots z_s^{c_s}$ is divisible by $d$, proving that $\gamma_K$ is not surjective. 

\medskip
To prove the claim, let :
 $$(y_1,\ldots,y_s,x_1,\ldots,x_q)\in \prod_{i=1}^s\op{Nrd}((A_i)_K^*)\times (K^*)^q,$$
 and 
 $$(z_1,\ldots,z_s)=\gamma_K(y_1,\ldots,y_s,x_1,\ldots,x_q).$$ 
 We have : 
$$z_1^{c_1}\ldots z_s^{c_s}=y_1^{c_1}\ldots y_s^{c_s}x_1^{\sum_{i=1}^s c_i b_{i,1}}\ldots x_q^{\sum_{i=1}^s c_i b_{i,q}}.$$
For every $i$ from $1$ to $s$ the integer $v(y_i)$ is divisible by $d_i$, as $y_i$ is a reduced norm of the central simple algebra $(A_i)_K$ over $K$ and the 
index of $A_i$ over $k$ is $d_i$. Therefore $v(y_i^{c_i})$ is divisible by $c_id_i$ , hence by $d$.
Moreover, for every index $j$ between $1$ and $q$, the sum $\sum_{i=1}^s c_i b_{i,j}$ is also divisible by $d$, completing the proof of the claim.

\end{proof}
Here is an example of a situtation where condition {\it(2)} in Proposition \ref{inner} is easy to work out.
Suppose that the group $R'_G$ decomposes along the direct factor decomposition $(*)$ in {\it(1)}. That is,
$$R'_G\simeq \mu_{m_1}\times\cdots\times \mu_{m_r}$$
where, for each index $i$, $\mu_{m_i}$ is a subgroup of $Z_{G_i}$. In this setting, condition $(2)$ in Proposition \ref{inner} is equivalent to the fact 
that for every $i$ from $1$ to $s$ the integers $d_i$ and $\frac{n_i}{m_i}$ are relatively prime.

 %To lighten the notations, we will denote by $Z$ the center of $G$, $R$ the radical of $G$, $Z'$ the center of $G'$, 
%and $R'$ the intersection of $R$ and $G'$ - which is not the radical of $G'$ in general. As these groups are isomorphic to the corresponding groups for the split form 
%$G_{\op{split}}$ of $G$ we see that $Z$, $Z'$ and $R'$ are split diagonalizable groups and $R$ is a split torus.  

%\medskip

\subsection{The classification of special quasisplit groups}
Recall that a reductive group $G$ over a field $k$ is called {\bf quasisplit} if it possesses a Borel subgroup defined over $k$, see for instance \cite[III, $2.2$]{Coho}, or, 
in other words, if the variety of Borel subgroups of $G$ has a rational point. We show now that a quasisplit group is special if and only if its derived subgroup and coradical 
are special.  
\begin{prop}\label{qsplit}
 Let $G$ be a reductive algebraic group over $k$. Then $G$ is quasisplit and special if and only if there exists an exact sequence of algebraic groups :
 $$ \begin{tikzpicture}
\node (A) at (0,0) {$1$};
\node (B) at (1.2,0) {$D$};
\node (C) at (3,0) {$G$};
\node (D) at (4.8,0) {$C$};
\node (E) at (6,0) {$1$};
\draw[->,>=latex] (A) to (B);
\draw[->,>=latex] (B) to (C);
\draw[->,>=latex] (C) to (D);
\draw[->,>=latex] (D) to (E);
\end{tikzpicture}$$
 where $D$ is isomorphic to a direct product :
 $$R_{K_1|k}(G_1)\times R_{K_2|k}(G_2)\times\cdots\times R_{K_r|k}(G_r)$$
 where, for every index $i$, $K_i$ is a finite separable extension of $k$, $G_i$ is equal to either $\op{SL}_{n_i}$ 
 or $\op{Sp}_{2n_i}$ for some integer $n_i$, and the group $C$ is a special torus over $k$.
In that case, $D$ is the derived subgroup of $G$ and $C$ is the coradical of $G$.
 \end{prop}
\begin{proof}
If such an exact sequence exists, then, as $D$ and $C$ are special, it follows readily from the derived exact sequence of pointed sets in fppf-cohomology that $G$ is special
as well. Moreover, as $C$ is commutative, the derived subgroup $G'$ is contained in $D$. Now, as $D$ is semisimple, it is equal to its own derived subgroup, and is in particular
contained in $G'$. Finally, we see that $D$ is equal to $G'$, and the fact that $C$ is the coradical of $G$ follows readily. 
Now, as $G$ and $G'$ share the same variety of Borel subgroups, and $G'$ is quasisplit, we see that $G$ is quasisplit as well.

\medskip
Suppose now that $G$ is quasisplit and special. By the same argument as above, the derived subgroup $G'$ is quasisplit. Moreover, by Proposition \ref{derived}, $G'$ is isomorphic to 
 $$R_{K_1|k}(G_1)\times R_{K_2|k}(G_2)\times\cdots\times R_{K_r|k}(G_r)$$
where, for each index $i$, the extension $K_i$ of $k$ is finite and separable and the group $G_i$ is isomorphic over $K_i$ to either $\op{SL}_1(A_i)$, where 
$A_i$ is a central simple algebra over $K_i$, or $\op{Sp}_{2n_i}$ for some integer $n_i$. For each index $i$ such that $G_i$ is isomorphic to 
$\op{SL}_1(A_i)$, we see that $\op{SL}_1(A_i)$ is a direct factor of $G_{K_i}$. As $G_{K_i}$ is quasisplit, this forces $\op{SL}_1(A_i)$ to be quasisplit as well, 
implying that $A_i$ is split. By Proposition \ref{coradical}, the coradical of $G$ is special. We thus have an exact sequence as in the proposition, with $D$ the derived 
subgroup $G'$ of $G$ and $C$ the coradical $C_G$.
\end{proof}

Let $G$ be an arbitrary reductive group over the field $k$. There is a unique inner form of $G$ that is quasisplit, called the {\bf quasisplit form} $G_{\op{qsplit}}$
of $G$, see for instance \cite{Sp}. 
In the following proposition, we prove that the quasisplit form of $G$ is special if $G$ is special. This is very reasonable, as we expect $G_{\op{qsplit}}$ 
to be less ``twisted'' than $G$.

\begin{prop}
Let $G$ be a special reductive group over $k$. The quasisplit form of $G$ is special as well.  
\end{prop}
\begin{proof}
The groups $G$ and $G_{\op{qsplit}}$ share the same coradical, and $(G_{\op{qsplit}})'$ is the quasisplit form of $G'$. Therefore, 
by Proposition \ref{derived} and \ref{coradical}, the derived subgroup and coradical of $G_{\op{qsplit}}$ are special, proving that $G_{\op{qsplit}}$ is special.

\end{proof}

\section{Special tori}\label{tori}
In this section we give the classification of special tori after Colliot-Thélène. This classification is implicitely contained in \cite{CTS} and 
explicitely given in the first version of \cite{BR} on the ArXiv but not in the published version. For this reason we thought that it would be a good idea to 
include it in the present paper. We actually reproduce the proof from \cite{BR}.
The relevance of the classification of special tori for our problem of classifying reductive groups is twofold : firstly, tori are reductive groups and secondly, 
by Proposition \ref{coradical}, the coradical of a special reductive group is a special torus.

\medskip

Let $k$ be a base field, $k_s$ a fixed separable closure and $\Gamma$ the absolute Galois group $\op{Gal}(k_s|k)$ of $k$. 
A continuous $\Gamma$-module is called a {\bf permutation $\Gamma$-module} if it is a free $\mathbb{Z}$-module possessing a basis which is 
permuted by the action of $\Gamma$. A continuous $\Gamma$-module is called {\bf invertible} if it is a direct factor of a permutation 
$\Gamma$-module.

\begin{thm}[Colliot-Thélène]\label{specialtori}
Let $T$ be a torus defined over a field $k$. The torus $T$ is special if and only if 
the character lattice of $T_{k_s}$ is invertible.
\end{thm}
\begin{proof}
If the character lattice of $T_{k_s}$ is invertible, then $T$ is a direct factor of a finite 
product of tori of type $R_{K|k}(\mathbb{G}_{m,K})$, where $K$ is a finite separable extension of $k$ and $\mathbb{G}_{m,K}$ is the multiplicative group over $K$. 
By Hilbert's theorem $90$ and Schapiro's lemma, 
it follows that $T$ is special.

\medskip

Conversely, assume that $T$ is special. Let $K$ be a finite separable field extension of $k$. By Lemma \ref{gille} below, 
$$H^1(K((t)),T)\simeq H^1(K,T)\oplus H^1(K,N)$$
where $N$ is the cocharacter lattice of $T_{k_s}$ and $K((t))$ is the field of formal Laurent series over $K$. Since the torus $T$ is special, we see that 
$H^1(K,N)$ is trivial. As this property holds for every finite separable field extension $K$ of $k$, it means that the torus $T$ is flasque.
By \cite[Proposition $7.4$]{CTS} a flasque torus is special if and only if the character lattice of $T_{k_s}$ is invertible, which completes the proof, modulo the 
following lemma :

\begin{lemme}\label{gille}
 For any torus over the field $k$, there is an isomorphism :
 $$H^1(k((t)),T)\simeq H^1(k,T)\oplus H^1(k,N)$$
 where $N$ is the cocharacter lattice of $T_{k_s}$.
\end{lemme}
{\it Proof of Lemma \ref{gille}}. Set $K=k((t))$. Let $L$ be the union of the field $k'((t))$ for all finite extensions of $k$ inside $k_s$. Then the Galois group 
$\op{Gal}(L|K)$ is equal to $\Gamma$. We have the inflation-restriction exact sequence, see for example \cite[I.2.6(b)]{Coho}:
$$ \begin{tikzpicture}

\node (B) at (3,0) {$1$};
\node (C) at (5.5,0) {$H^1(\Gamma,T(L))$};
\node (D) at (9,0) {$H^1(K,T)$};
\node (E) at (12,0) {$H^1(L,T)$};
\node (F) at (14.3,0) {$1$};

%\node (F) at (10.5,0.2) {$\gamma_K$};
%\node (G) at (1.5,0.2) {$\alpha_K$};

\draw[->,>=latex] (B) to (C);
\draw[->,>=latex] (C) to (D);
\draw[->,>=latex] (D) to (E);
\draw[->,>=latex] (E) to (F);
\end{tikzpicture}$$
The torus $T$ is split over $L$, hence, by Hilbert's theorem $90$, we see that the sets $H^1(\Gamma,T(L))$ and $H^1(K,T)$ are equal. 
By \cite[Lemma $5.17(3)$]{ChGP}, we have 
$$H^1(\Gamma,T(L))\simeq H^1(\Gamma,T(k_s[t,t^{-1}])).$$
Now, we write :
$$T(k_s[t,t^{-1}])=N\otimes_{\mathbb{Z}} k_s[t,t^{-1}]^*=N\otimes_{\mathbb{Z}}(k_s^*\oplus \mathbb{Z})=T(k_s)\oplus N$$
because $T$ splits over $k_s$ and $k_s[t,t^{-1}]^*$ is equal to $k_s^*\oplus \mathbb{Z}$. Finally, we obtain :
$$H^1(K,T)\simeq H^1(\Gamma,T(L))\simeq H^1(\Gamma,T(k_s[t,t^{-1}]))\simeq$$
$$ H^1(\Gamma,T(k_s))\oplus H^1(\Gamma,\oplus N)\simeq H^1(k,T)\oplus H^1(k,\oplus N).$$
 
\end{proof}

\bibliographystyle{plain} 
\bibliography{mat}

%Department of Mathematics, University of British Columbia, 1984 Mathematics Road, Vancouver, BC, V6T 1Z2, Canada. huruguen@math.ubc.ca

\end{document}